\documentclass[11pt,a4paper,leqno]{amsart}
\pdfoutput=1

\usepackage[latin1]{inputenc}
\usepackage[T1]{fontenc}
\usepackage{amsfonts}
\usepackage{amsmath}
\usepackage{amssymb}
\usepackage{eurosym}
\usepackage{mathrsfs}
\usepackage{palatino}
\usepackage{color}
\usepackage{xcolor}
\usepackage{esint}
\usepackage{url}
\usepackage{graphicx}

\numberwithin{equation}{section}

\newcommand{\ud}[0]{\,\mathrm{d}}

\newcommand{\ave}[1]{\langle #1\rangle}

\theoremstyle{plain}
\newtheorem{thm}{Theorem}[section]
\newtheorem{lem}[thm]{Lemma}

\theoremstyle{definition}

\theoremstyle{remark}
\newtheorem{rem}[thm]{Remark}

\title{Weak type $A_p$ estimate for bilinear Calder\'on-Zygmund operators}

\author{Linfei Zheng}

\address{Center for Applied Mathematics, Tianjin University, Weijin Road 92, 300072 Tianjin, China}
\email{linfei\_zheng@tju.edu.cn}

\makeatletter
\@namedef{subjclassname@2020}{%
  \textup{2020} Mathematics Subject Classification}
\makeatother


\begin{document}

\allowdisplaybreaks

\begin{abstract}

In this paper, we investigate the boundedness of bilinear Calder\'on-Zygmund operators $T$ from ${L^{p_1}\left(w_1\right)} \times {L^{p_2}\left(w_2\right)}$ to ${L^{p,\infty}\left(v_{\vec{w}}\right)}$ with the stopping time method, where $1 / p = 1 / p_1 + 1 / p_2$ , $1 < p_1, p_2 < \infty$ and $\vec{w}$ is a multiple $A_{\vec{P}}$ weight. Specifically, we studied the exponent $\alpha$ of $A_{\vec{P}}$ constant in formula $$\|T(\vec{f})\|_{L^{p,\infty}\left(v_{\vec{w}}\right)} \leqslant C_{m, n, \vec{P}, T}[\vec{w}]_{A_{\vec{P}}}^{\alpha}\left\|f_1\right\|_{L^{p_1}\left(w_1\right)}\left\|f_2\right\|_{L^{p_2}\left(w_2\right)}.$$ Surprisingly, we show that when $p \geqslant \frac{3+\sqrt{5}}{2}$ or $\min\{p_1,p_2\} > 4$, the index $\alpha$ in the above equation can be less than $1$, which is different from the linear scenario.
\end{abstract}

\maketitle

\section{Introduction and main results}
In recent years, the theory of Calder\'on-Zygmund operators has attracted widespread attention. There have been many advances in the optimal control of weighted operator norms with $A_p$ weights.
 
In the linear case, in 2012, Hyt\"{o}nen proved the $A_2$ conjecture in \cite{Hyt} and obtained 
$$\|T(f)\|_{L^{p}(w)} \lesssim [w]_{A_p}^{\max\{1,\frac{p^{\prime}}{p}\}}\left\|f\right\|_{L^{p}(w)}.$$
Just one year later, in \cite{Ler}, Lerner proved that the Calder\'on-Zygmund operators can be controlled by sparse operators and provided an alternative proof of the $A_2$ theorem. We recommend interested readers to learn about the history of $A_2$ theorem in the above two papers and the references therein. 

For the weak weighted operator norms of the Calder\'on-Zygmund operators, in \cite{HL}, Hyt\"{o}nen and Lacey obtained a mixed $A_p-A_\infty$ estimate. If we represent it with $A_p$ weight only, we can obtain
$$\|T(f)\|_{L^{p,\infty}(w)} \lesssim [w]_{A_p}\left\|f\right\|_{L^{p}(w)}.$$

In the multilinear case, in \cite{LMS}, Li, Moen, and Sun proved that when $1 < p, p_1,p_2 < \infty$,
$$
\|T(\vec{f})\|_{L^{p}\left(v_{\vec{w}}\right)} \lesssim [\vec{w}]_{A_{\vec{P}}}^{\max\{1,\frac{p_1^{\prime}}{p},\frac{p_2^{\prime}}{p}\}}\left\|f_1\right\|_{L^{p_1}\left(w_1\right)}\left\|f_2\right\|_{L^{p_2}\left(w_2\right)},
$$ 
and they provided a beautiful example to show that their result is optimal. When $p < 1$, the above inequality still holds since the Calder\'on-Zygmund operators can be controlled by sparse operators pointwise, as shown in \cite{CR,LN,DHL}. For weak norms, it is generally believed that the optimal index in $A_{\vec{P}}$ estimate is 1, which is the same as the linear case. Li and Sun gave a mixed $A_p-A_\infty$ estimate in \cite{LS}, but if we represent it with $A_p$ weight as the linear case, the index is greater than $1$. In Li's master's thesis \cite{Li}, he used the Coifman-Fefferman inequality to prove that the index in weak type $A_{\vec{P}}$ estimate can be $1+1/p$. 

In this paper, we show that the index in weak $A_{\vec{P}}$ estimate for bilinear Calder\'on-Zygmund operators can be less than 1 for certain $\vec{P}$. Specifically, we prove the following theorem.
\begin{thm}\label{main theorem}
Let $T$ be a bilinear Calder\'on-Zygmund operator, $\vec{P}=\left(p_1, p_2\right)$ with $1 / p_1 + 1/ p_2 = 1 / p$ and $1 < p, p_1,p_2 < \infty$. Suppose $\vec{w}=\left(w_1, w_2 \right) \in A_{\vec{P}}$, then 
\begin{equation}
\|T(\vec{f})\|_{L^{p,\infty}\left(v_{\vec{w}}\right)} \leqslant C_{m, n, \vec{P}, T} [\vec{w}]_{A_{\vec{P}}}^{\alpha}\left\|f_1\right\|_{L^{p_1}\left(w_1\right)}\left\|f_2\right\|_{L^{p_2}\left(w_2\right)}, \label{formula in the main theorem}
\end{equation}
where $$\alpha = \min{\{\beta,\gamma\}}, \beta=\frac{1}{p} + \max{\big\{\min\big\{\frac{1}{p_1^{\prime}},\frac{1}{p_1^{\prime}}\frac{p_2^{\prime}}{p}\big\}, \min\big\{\frac{1}{p_2^{\prime}}, \frac{1}{p_2^{\prime}}\frac{p_1^{\prime}}{p}\big\} \big\}}, \gamma=\max\{1,\frac{p_1^{\prime}}{p},\frac{p_2^{\prime}}{p}\}.\ $$
\end{thm}
It should be noted that the exponent $\gamma$ in the theorem comes from the strong type estimate mentioned above, so we only need to prove equation (\ref{formula in the main theorem}) with exponent $\alpha=\beta$ in the following.

\begin{rem}
Note that $\beta \leqslant \max{\big\{\frac{1}{p}+\frac{1}{p_1^{\prime}},\frac{1}{p}+\frac{1}{p_2^{\prime}}\big\}} < 1+\frac{1}{p}$, so our results improve the one by \cite{Li}. 
\end{rem}
\begin{rem}
We can apply the extrapolation techniques demonstrated in \cite[Theorem 4.1]{Zoe}, to generalize Theorem\ref{main theorem} to the case of $p < 1$. In paricular, if we use $\vec{P}=\left(p_1, p_2\right)$ with $2 < p_ 1=p_ 2 < \sqrt{2}+1$ as the starting point, we can obtain better results than strong type, and further details are left for interested readers. 
\end{rem}

\section{Preliminaries}
\subsection{Bilinear Calder\'on-Zygmund operators}
We call $T$ a \emph{bilinear Calder\'on-Zygmund operator} if it is originally defined on the product of Schwartz spaces and takes values in tempered distributions, meanwhile, for some $1 < q_1 , q_2 < \infty $, it can be extended into a bounded bilinear operator from $L^{q_1}(\mathbb{R}^n) \times L^{q_2}(\mathbb{R}^n)$ to $L^{q}(\mathbb{R}^n)$, where $1 / q_1 + 1/ q_2 = 1 / q$, and if there exists a function $K(y_0,y_1,y_2)$, defined off the diagonal $y_0 = y_1 = y_2 $ in $(\mathbb{R}^n)^3$ that satisfies 
$$T(f_1,f_2)(y_0)=\int_{(\mathbb{R}^n)^2}K(y_0,y_1,y_2)f_1(y_1)f_2(y_2)\ud y_1\ud y_2, \quad \forall y_0 \notin \operatorname{supp}f_1 \cap \operatorname{supp}f_2;$$
$$|K(y_0,y_1,y_2)| \leqslant \frac{C}{(|y_0 - y_1| + |y_0 - y_2|)^{2n}};$$ and for some $A, \varepsilon > 0$, whenever $|h| \leqslant \frac{1}{2}\max\{|y_0-y_1|,|y_0-y_2|\}$,
\begin{align}
|K(y_0 + h,y_1,y_2) - K(y_0,y_1,y_2)| &+ |K(y_0,y_1+h,y_2) - K(y_0,y_1,y_2)|\nonumber\\
&+|K(y_0,y_1,y_2+h) - K(y_0,y_1,y_2)|\nonumber\\ 
&\leqslant \frac{1}{(|y_0 - y_1| + |y_0 - y_2|)^{2n}}\omega \bigg(\frac{|h|}{|y_0 - y_1| + |y_0 - y_2|}\bigg),\nonumber
\end{align}
where $\omega$ is a modulus of Dini-continuity, in other words, an increasing function satisfies $\omega(0)=0$, $\omega(t+s) \leqslant \omega(t)+\omega(s)$, and $$\|\omega\|_{\text{Dini}}:=\int_{0}^{1}\omega(t)\frac{\ud t}{t} < \infty.$$In \cite{DHL}, Dimi\'an, Hormozi, and Li proved that bilinear Calder\'on-Zygmund operators can be pointwise controlled by sparse operators, which will be introduced in section \ref{sparse operators}.

\subsection{Multiple $A_{\vec{p}}$ weight}

Recall that in the linear case, a \emph{weight} is a non-negative locally integrable function. When $1<p<\infty$, The set $A_p$ is composed of weights that satisfy 
$$[w]_{A_{p}}:=\sup_{Q: \,\text{cube in}\, \mathbb{R}^n} \ave{w}_Q\ave{w^{1-p^{\prime}}}_Q^{p-1}<\infty ,
$$
where $\ave{w}_Q := w(Q)/|Q|$. When $p=\infty$, $A_{\infty}:=\bigcup_{1<p<\infty} A_p$ and the $A_{\infty}$ constant $[w]_{A_{\infty}}$ is defined by
$$
[w]_{A_{\infty}}:=\sup_Q \frac{1}{w(Q)} \int_{Q} M\left(w \chi_Q\right),
$$where $M$ denotes the Hardy-Littlewood maximal functions. Meanwhile, for any $w\in A_p,$ we have $[w]_{A_{\infty}} \leqslant [w]_{A_{p}}$. 

As is well known, in \cite{LOPTT}, Lerner, Ombrosi, P\'erez, Torres, and Trujillo-Gonz\'alez extended the above definition to the multilinear case, and defined \emph{multiple $A_{\vec{p}}$ weights} as follows. Let $\vec{P}=\left(p_1, \cdots, p_m\right)$ with $1 < p_1, \ldots, p_m < \infty$ and $1 / p_1+\cdots+1 / p_m=1 / p$. Given $\vec{w}=\left(w_1, \cdots, w_m\right)$, set
$$
v_{\vec{w}}=\prod_{i=1}^m w_i^{p / p_i},
$$
the $A_{\vec{P}}$ constant is defined by
$$
[\vec{w}]_{A_{\vec{P}}}:=\sup_Q \ave{v_{\vec{w}}}_Q \prod_{i=1}^m \ave{\sigma_i}_Q^{p / p_i^{\prime}},
$$
where $\sigma_i = w_i^{1-p_i^{\prime}}$. We say that $\vec{w}$ satisfies the multilinear $A_{\vec{P}}$ condition if $[\vec{w}]_{A_{\vec{P}}} < \infty$. Particularly, in Theorem 3.6 of the aforementioned paper, they proved that 
\begin{equation} 
[v_{\vec{w}}]_{A_{mp}} \leqslant [\vec{w}]_{A_{\vec{P}}}, [\sigma_i]_{A_{mp_i^{\prime}}} \leqslant [\vec{w}]_{A_{\vec{P}}}^{p_i^{\prime}/p}, \quad \forall \vec{w} \in A_{\vec{P}}\label{important formula}.
\end{equation}
\subsection{Dyadic cubes system, sparse operators and stopping time argument}\label{sparse operators}

The \emph{dyadic cubes system} $\mathscr{D}$ is a family of cubes with the following properties:

(1) for any $Q \in \mathscr{D}$, its sides are parallel to the coordinate axes and its sidelength is of the form of $2^k$.

(2) $Q\cap R \in \{Q,R,\emptyset\}$, for any $Q,R \in \mathscr{D}$.

(3) the cubes of fixed sidelength $2^k$ form a partition of $\mathbb{R}^n$

A collection $\mathcal{S} \subset \mathscr{D}$ is called \emph{sparse} if for each $Q \in \mathcal{S}$, there exists a subset $E_Q \subset Q$ such that $|E_Q| \geqslant \frac{1}{2}|Q|$ and the sets $\{E_Q\}_{Q \in \mathcal{S}}$ are pairwise disjoint. For a sparse family $\mathcal{S}$, we can define the \emph{sparse operator} $A_{\mathscr{D}, \mathcal{S}}$ as follows.
$$A_{\mathscr{D}, \mathcal{S}}(\vec{f})=\sum\limits_{Q \in \mathcal{S}}\ave{f_1}_Q\ave{f_2}_Q\chi_Q,$$
where $\vec{f}=(f_1,f_2)$. 

Below, we will introduce the main technique of this paper, stopping time argument, which was introduced by Li and Sun in \cite{LS}, and further improved by Dimi\'an, Hormozi, and Li in \cite{DHL}.

Let $w$ be a weight and $f \in L^p(w)$ for some $0 < p < \infty$, suppose that the sparse family $\mathcal{S}$ has a collection of maximal cubes, in other words, there exists a collecion of disjoint cubes $\{Q_{i}\}_{i \in \Lambda} \subset \mathcal{S}$, such that for any cube $Q \in \mathcal{S}$, there exists $i \in \Lambda$, $Q \subset Q_{i}$. Now we construct the \emph{stopping time family} $\mathscr{F}$ from pair $(f,w)$. Let $\mathscr{F}_0 := \{Q_{i}\}_{i \in \Lambda}$ and
$$\mathscr{F}_{k} := \bigcup_{F \in \mathscr{F}_{k-1}}\{F^{\prime} \subset F : F^{\prime} \text{\, is the maximal cube in\,} \mathcal{S} \text{\,that satisfies\,} \left\langle f \right\rangle_{F^{\prime}}^{w} > 2 \left\langle f\right\rangle_{F}^{w}\},$$ 
where $\left\langle f \right\rangle_{Q}^{w} := \int_{Q}fw \,dx/ w(Q)$, then stopping time family $\mathscr{F}:=\bigcup_{k=0}^ {\infty}\mathscr{F}_k$. It is easy to deduce from the above construction that  
\begin{equation}
\sum\limits_{F \in \mathscr{F}}\left(\left\langle f \right\rangle_{F}^{w}\right)^{p}w(F) \lesssim \|\operatorname{M}_{\mathscr{D}}^{w}(f)\|_{L^{p}\left(w\right)}^p \lesssim \left\|f\right\|_{L^{p}\left(w\right)}^p, \label{basic formula}
\end{equation}
where $\operatorname{M}_{\mathscr{D}}^{w}(f)(x):=\sup_{x \in Q, Q \in \mathscr{D}} \left\langle f \right\rangle_{Q}^{w}$. We use $\pi_{\mathscr{F}}(Q)$ to represent the \emph{stopping parents} of $Q$, that is, the minimal cube containing $Q$ in $\mathscr{F}$. According to the definition, we have $\left\langle f \right\rangle_{Q}^{w} \leqslant 2\left\langle f \right\rangle_{\pi_{\mathscr{F}}(Q)}^{w}$.

\section{Proof of the main results}
In order to prove the main theorem, we need the following lemmas.
\begin{lem}\emph{(}\cite[lemma 3.1]{LMS}\emph{)}\label{reference lemma1}
Let $\vec{P}=(p_1, p_2)$ with $1/p_1 + 1/p_2 = 1/p$ and $1 < p, p_1, p_2 < \infty$, $\vec{w}=\left(w_1, w_2\right) \in A_{\vec{P}}$. Then $\vec{w}^{1}=(v_{\vec{w}}^{1-p^{\prime}}, w_2) \in A_{\vec{P}^1}$, with $\vec{P}^{1}=\left(p^{\prime},p_2\right)$ and
$$[\vec{w}^{1}]_{A_{\vec{P}^{1}}}=[\vec{w}]_{A_{\vec{P}}}^{p_1^{\prime}/p}.$$
\end{lem}

\begin{lem}\emph{(}\cite[lemma 4.5]{LS}\emph{)}\label{reference lemma2}
Let $\vec{P}=(p_1, p_2)$ with $1/p_1 + 1/p_2 = 1/p$ and $1 < p, p_1, p_2 < \infty$, $\vec{w}=\left(w_1,w_2\right) \in A_{\vec{P}}$. Suppose that $\mathscr{D}$ is a dyadic cubes system and $\mathcal{S}$ is a sparse family in $\mathscr{D}$. Then the following assertions are equivalent. 

(1)$\|A_{\mathscr{D}, \mathcal{S}}(|f_1|\sigma_1,|f_2|\sigma_2)\|_{L^{p,\infty}\left(v_{\vec{w}}\right)} \leqslant C \prod_{i=1}^2\left\|f_i\right\|_{L^{p_i}\left(\sigma_i\right)}.$

(2)$\int_{Q}A_{\mathscr{D}, \mathcal{S}}(|f_1|\sigma_1\chi_{Q},|f_2|\sigma_2\chi_{Q})v_{\vec{w}}\ud x \leqslant C \prod_{i=1}^2\left\|f_i\right\|_{L^{p_i}\left(\sigma_i\right)}v_{\vec{w}}(Q)^{1/p^{\prime}}$ for all dyadic cubes $Q \in \mathcal{S}$ and all functions $f_i \in L^{p_i}\left(\sigma_i\right)$, $i=1,2.$
\end{lem}

\begin{rem}Review the proof of the above lemma in article \cite{LS}, we found that although the constant $C$ in the two equivalent propositions may be different, they can be compared to each other.
\end{rem}

\begin{lem}\label{main lemma}
Let $\vec{P}=(p_1, p_2)$ with $1/p_1 + 1/p_2 = 1/p$ and $1 < p, p_1, p_2 < \infty$, $\vec{w}=\left(w_1,w_2\right) \in A_{\vec{P}}$. Suppose that $\widetilde{Q}$ is a dyadic cube and $\operatorname{supp}{f_2} \subset \widetilde{Q}$, then 

\begin{align}
\|\chi_{\widetilde{Q}}A_{\mathscr{D}, \mathcal{S}}(\sigma_1\chi_{\widetilde{Q}}, |f_2|\sigma_2)\|_{L^{p}\left(v_{\vec{w}}\right)} \lesssim & \max{\big\{\min{\{[\sigma_1]_{A_\infty},[\sigma_2]_{A_\infty}\}}^{1/p},\min{\{[\sigma_1]_{A_\infty},[v_{\vec{w}}]_{A_\infty}\}}^{1/p_2^{\prime}}\big\}} \nonumber\\
& \times [\vec{w}]_{A_{\vec{P}}}^{1/p}\left\|f_2\right\|_{L^{p_2}\left(\sigma_2\right)}\sigma_1(\widetilde{Q})^{1/p_1}. \nonumber
\end{align}
\end{lem}

Suppose Lemma \ref{main lemma} is proven, referring to the method in \cite{LS}, we can directly prove Theorem \ref{main theorem} as follows.
\begin{proof}[Proof of Theorem\ref{main theorem}]
Using Lemma \ref{reference lemma1} and Lemma \ref{main lemma}, for each $Q \in \mathcal{S}$, we have 
\begin{align}
& v_{\vec{w}}(Q)^{-1/p^{\prime}}\int_{Q}A_{\mathscr{D}, \mathcal{S}}(|f_1|\sigma_1\chi_{Q},|f_2|\sigma_2\chi_{Q}) v_{\vec{w}}\ud x \nonumber\\
= & v_{\vec{w}}(Q)^{-1/p^{\prime}}\int_{Q}A_{\mathscr{D}, \mathcal{S}}(v_{\vec{w}}\chi_{Q},|f_2|\sigma_2\chi_{Q})|f_1|\sigma_1\ud x \nonumber\\
\leqslant & v_{\vec{w}}(Q)^{-1/p^{\prime}}\left(\int_{Q}\left(A_{\mathscr{D}, \mathcal{S}}(v_{\vec{w}}\chi_{Q},|f_2|\sigma_2\chi_{Q})\right)^{p_1^{\prime}}\sigma_1\ud x\right)^{1/p_1^{\prime}}\left(\int_{Q}|f_1|^{p_1}\sigma_1\ud x\right)^{1/p_1} \nonumber\\
\lesssim & \max{\big\{\min{\{[v_{\vec{w}}]_{A_\infty},[\sigma_2]_{A_\infty}\}}^{\frac{1}{p_1^{\prime}}},\min{\{[v_{\vec{w}}]_{A_\infty},[\sigma_1]_{A_\infty}\}}^{\frac{1}{p_2^{\prime}}}\big\}}[\vec{w}^{1}]_{A_{\vec{P}^1}}^{\frac{1}{p_1^{\prime}}}\left\|f_1\right\|_{L^{p_1}\left(\sigma_1\right)}\left\|f_2\right\|_{L^{p_2}\left(\sigma_2\right)} \nonumber\\
\overset{(\ref{important formula})}{\leqslant} & [\vec{w}]_{A_{\vec{P}}}^{\frac{1}{p} + \max{\big\{\min\big\{\frac{1}{p_1^{\prime}},\frac{1}{p_1^{\prime}}\frac{p_2^{\prime}}{p}\big\}, \min\big\{\frac{1}{p_2^{\prime}}, \frac{1}{p_2^{\prime}}\frac{p_1^{\prime}}{p}\big\}\big\}}}\left\|f_1\right\|_{L^{p_1}\left(\sigma_1\right)}\left\|f_2\right\|_{L^{p_2}\left(\sigma_2\right)}. \nonumber
\end{align}
Finally, according to Lemma \ref{reference lemma2}, we get the desired result. This finishes the proof.
\end{proof}

To prove Lemma \ref{main lemma}, we need the following lemma.
\begin{lem}\emph{(}\cite[lemma 4.15]{DHL}\emph{)}\label{reference lemma3}
Let $\vec{P}=(p_1, p_2)$ with $1/p_1 + 1/p_2 = 1/p$ and $1 < p, p_1, p_2 < \infty$, $\vec{w}=\left(w_1,w_2\right) \in A_{\vec{P}}$. Then for any sparse family  $\mathcal{S}$, we have 
\begin{equation}
\bigg\|\sum\limits_{Q \in \mathcal{S}}\ave{\sigma_1}_Q\ave{\sigma_2}_Q \chi_Q\bigg\|_{L^{p}\left(v_{\vec{w}}\right)} \lesssim [\vec{w}]_{A_{\vec{P}}}^{1/p}\bigg(\sum\limits_{Q \in \mathcal{S}}\ave{\sigma_1}_Q^{p/p_1}\ave{\sigma_2}_Q^{p/p_2}|Q|\bigg)^{1/p} \label{key formula1}
\end{equation}
\begin{equation}
\bigg\|\sum\limits_{Q \in \mathcal{S}}\ave{\sigma_1}_Q\ave{v_{\vec{w}}}_Q \chi_Q\bigg\|_{L^{p_2^{\prime}}\left(\sigma_2\right)} \lesssim
 [\vec{w}]_{A_{\vec{P}}}^{1/p}\bigg(\sum\limits_{Q \in \mathcal{S}}\ave{\sigma_1}_Q^{p_2^{\prime}/p_1}\ave{v_{\vec{w}}}_Q^{p_2^{\prime}/p^{\prime}}|Q|\bigg)^{1/p_2^{\prime}} \label{key formula2}
\end{equation}
\end{lem}

\begin{proof}[Proof of Lemma \ref{main lemma}]
In the first half of the proof, we will use a method similar to that in \cite{DHL} and \cite{LS}. Since $\operatorname{supp}{f_2} \subset \widetilde{Q}$, we have 
\begin{align}
A_{\mathscr{D}, \mathcal{S}}(\sigma_1\chi_{\widetilde{Q}}, |f_2|\sigma_2) & = \sum_{\substack{Q \in \mathcal{S}\\ Q \cap {\widetilde{Q}} \neq \emptyset}}\ave{\sigma_1\chi_{\widetilde{Q}}}_Q \ave{|f_2|\sigma_2}_Q \chi_Q \nonumber\\
& = \sum_{\substack{Q \in \mathcal{S}\\ {\widetilde{Q}} \subset Q}}\ave{\sigma_1\chi_{\widetilde{Q}}}_Q \ave{|f_2|\sigma_2}_Q\chi_Q + \sum_{\substack{Q \in \mathcal{S} \\ Q \subset {\widetilde{Q}}}}\ave{\sigma_1}_Q \ave{|f_2|\sigma_2}_Q \chi_Q \nonumber\\
& := A_{\mathscr{D}, \mathcal{S}}^{1}(\sigma_1\chi_{\widetilde{Q}}, |f_2|\sigma_2)+A_{\mathscr{D}, \mathcal{S}}^{2}(\sigma_1\chi_{\widetilde{Q}}, |f_2|\sigma_2).\nonumber
\end{align}
For $A_{\mathscr{D}, \mathcal{S}}^{1}(\sigma_1\chi_{\widetilde{Q}}, |f_2|\sigma_2)$, the calculation is not difficult.
\begin{align}
\bigg\|\chi_{\widetilde{Q}} A_{\mathscr{D}, \mathcal{S}}^{1}(\sigma_1\chi_{\widetilde{Q}}, |f_2|\sigma_2) \bigg\|_{L^{p}\left(v_{\vec{w}}\right)} &= \bigg\| \sum_{\widetilde{Q} \subset Q}\frac{\sigma_1(Q \cap \widetilde{Q})\int_{\widetilde{Q}}f_2(y_2)\sigma_2\ud y_2}{|Q|^2} \chi_{\widetilde{Q}}\bigg\|_{L^{p}\left(v_{\vec{w}}\right)}  \nonumber\\
& \lesssim \left\| \frac{\sigma_1(\widetilde{Q})\int_{\widetilde{Q}}f_2(y_2)\sigma_2\ud y_2}{|\widetilde{Q}|^2} \chi_{\widetilde{Q}}\right\|_{L^{p}\left(v_{\vec{w}}\right)} \nonumber\\
& \leqslant \frac{\sigma_1(\widetilde{Q})\left\|f_2\right\|_{L^{p_2}\left(\sigma_2\right)}\sigma_2(\widetilde{Q})^{1/p_2^{\prime}}}{|\widetilde{Q}|^2}v_{\vec{w}}(\widetilde{Q})^{1/p}\nonumber\\
& \leqslant [\vec{w}]_{A_{\vec{P}}}^{1/p}\left\|f_2\right\|_{L^{p_2}\left(\sigma_2\right)}\sigma_1(\widetilde{Q})^{1/p_1}. \nonumber
\end{align}
It remains to estimate $A_{\mathscr{D}, \mathcal{S}}^{2}(\sigma_1\chi_{\widetilde{Q}}, |f_2|\sigma_2)$. By duality, we have
\begin{align}
\left\| A_{\mathscr{D}, \mathcal{S}}^{2}(\sigma_1\chi_{\widetilde{Q}}, |f_2|\sigma_2) \right\|_{L^{p}\left(v_{\vec{w}}\right)} &= \bigg\| \sum_{Q \subset \widetilde{Q}}\ave{f_2}_{Q}^{\sigma_2}\ave{\sigma_1}_{Q}\ave{\sigma_2}_{Q} \chi_{Q} \bigg\|_{L^{p}\left(v_{\vec{w}}\right)} \nonumber\\
& = \sup_{\|h\|_{L^{p'}\left(v_{\vec{w}}\right)}=1}\sum_{Q \subset \widetilde{Q}}\ave{f_2}_{Q}^{\sigma_2}\ave{\sigma_1}_{Q}\ave{\sigma_2}_{Q}\int_{Q} h\ud v_{\vec{w}}   \nonumber\\
& = \sup_{\|h\|_{L^{p'}\left(v_{\vec{w}}\right)}=1}\sum_{Q \subset \widetilde{Q}}\ave{f_2}_{Q}^{\sigma_2}\ave{h}_{Q}^{v_{\vec{w}}}\ave{\sigma_1}_{Q}\ave{\sigma_2}_{Q}v_{\vec{w}}(Q). \nonumber
\end{align}

Let $\mathcal{S}^{'} = \mathcal{S} \cap \widetilde{Q}$, then $\widetilde{Q}$ is the maximal cube in the sparse family $\mathcal{S}^{'}$, we can use the stopping time argument mentioned above. Let $\mathscr{F}_2$ and $\mathscr{H}$ represent the stopping time family constructed by $(f_2,\sigma_2)$ and $(h,v_{\vec{w}})$ respectively, and write $\pi_{\mathscr{F}_2}(Q)= F_2 $, and $\pi_{\mathscr{H}}(Q)= H$ together as $\pi(Q)= (F_2, H)$. Then,
\begin{align}
\sum_{Q \in \mathcal{S}^{'}}\ave{f_2}_{Q}^{\sigma_2}\ave{h}_{Q}^{v_{\vec{w}}}\ave{\sigma_1}_{Q}\ave{\sigma_2}_{Q}v_{\vec{w}}(Q)=&\sum_{F_2 \in \mathscr{F}_2 }\sum_{\substack{H \in \mathscr{H} \\ H \subset F_2 }}\sum_{\substack{Q \in \mathcal{S}^{\prime} \\ \pi(Q)= (F_2, H)}} \ave{f_2}_{Q}^{\sigma_2}\ave{h}_{Q}^{v_{\vec{w}}}\lambda_{Q} \nonumber\\
& + \sum_{H \in \mathscr{H}}\sum_{\substack{F_2 \in \mathscr{F}_2 \\ F_2 \subset H }}\sum_{\substack{Q \in \mathcal{S}^{\prime} \\ \pi(Q)= (F_2, H)}}  \ave{f_2}_{Q}^{\sigma_2}\ave{h}_{Q}^{v_{\vec{w}}}\lambda_{Q} \nonumber\\
 := & I_1 + I_2, \nonumber
\end{align}

where $\lambda_Q  = \ave{\sigma_1}_{Q}\ave{\sigma_2}_{Q}v_{\vec{w}}(Q)$. For $I_1$, we have 
\begin{align}
I_1 \leqslant& \, 4 \sum_{F_2 \in \mathscr{F}_2 }  \ave{f_2}_{F_2}^{\sigma_2} \sum_{\substack{H \in \mathscr{H} \\ H \subset F_2 }} \ave{h}_{H}^{v_{\vec{w}}} \sum_{\substack{Q \in \mathcal{S}^{\prime} \\ \pi(Q)= (F_2, H)}}\lambda_{Q} \nonumber\\
\lesssim & \sum_{F_2 \in \mathscr{F}_2 }  \ave{f_2}_{F_2}^{\sigma_2} \int_{F_2} \sum_{\substack{H \in \mathscr{H} \\ H \subset F_2 }} \ave{h}_{H}^{v_{\vec{w}}} \sum_{\substack{Q \in \mathcal{S}^{\prime} \\ \pi(Q)= (F_2, H)}}\frac{\lambda_{Q}\chi_{Q}}{v_{\vec{w}}(Q)}\ud v_{\vec{w}} \nonumber\\
\lesssim & \sum_{F_2 \in \mathscr{F}_2 }  \ave{f_2}_{F_2}^{\sigma_2} \int_{F_2}\Big(\sup _{\substack{H^{\prime} \in \mathscr{H} \\
\pi_{\mathscr{F}_2}(H^{\prime})=F_2}} \ave{h}_{H^{\prime}}^{v_{\vec{w}}}\chi_{H^{\prime}}\Big)\sum_{\substack{H \in \mathscr{H} \\ H \subset F_2 }}  \sum_{\substack{Q \in \mathcal{S}^{\prime} \\ \pi(Q)= (F_2, H)}}\frac{\lambda_{Q}\chi_{Q}}{v_{\vec{w}}(Q)}\ud v_{\vec{w}}\nonumber\\
\lesssim & \sum_{F_2 \in \mathscr{F}_2 }  \ave{f_2}_{F_2}^{\sigma_2} \bigg\|\sum_{\substack{H \in \mathscr{H} \\ H \subset F_2 }}  \sum_{\substack{Q \in \mathcal{S}^{\prime} \\ \pi(Q)= (F_2, H)}}\frac{\lambda_{Q}\chi_{Q}}{v_{\vec{w}}(Q)} \bigg\|_{L^{p}\left(v_{\vec{w}}\right)}\bigg\|\sup _{\substack{H^{\prime} \in \mathscr{H} \\
\pi_{\mathscr{F}_2}(H^{\prime})=F_2}} \ave{h}_{H^{\prime}}^{v_{\vec{w}}}\chi_{H^{\prime}}\bigg\|_{L^{p^{\prime}}\left(v_{\vec{w}}\right)}\nonumber\\
\leqslant & \bigg(\sum_{F_2 \in \mathscr{F}_2 }  \left(\ave{f_2}_{F_2}^{\sigma_2}\right)^{p} \bigg\|\sum_{\substack{H \in \mathscr{H} \\ H \subset F_2 }}  \sum_{\substack{Q \in \mathcal{S}^{\prime} \\ \pi(Q)= (F_2, H)}}\frac{\lambda_{Q}\chi_{Q}}{v_{\vec{w}}(Q)} \bigg\|_{L^{p}\left(v_{\vec{w}}\right)}^{p}\bigg)^{\frac{1}{p}} \nonumber\\
& \times \bigg(\sum_{F_2 \in \mathscr{F}_2 }\sum _{\substack{H^{\prime} \in \mathscr{H} \\
\pi_{\mathscr{F}_2}(H^{\prime})=F_2}}\left(\ave{h}_{H^{\prime}}^{v_{\vec{w}}}\right)^{p^{\prime}}v_{\vec{w}}\left(H^{\prime}\right)\bigg)^{\frac{1}{p^{\prime}}}\nonumber\\
\lesssim & \bigg(\sum_{F_2 \in \mathscr{F}_2 }  \left(\ave{f_2}_{F_2}^{\sigma_2}\right)^{p} \bigg\|\sum_{\substack{H \in \mathscr{H} \\ H \subset F_2 }}  \sum_{\substack{Q \in \mathcal{S}^{\prime} \\ \pi(Q)= (F_2, H)}}\frac{\lambda_{Q}\chi_{Q}}{v_{\vec{w}}(Q)} \bigg\|_{L^{p}\left(v_{\vec{w}}\right)}^{p}\bigg)^{\frac{1}{p}}.\nonumber
\end{align}

The last inequality is due to (\ref{basic formula}). By (\ref{key formula1}), we have
\begin{align}
\bigg\|\sum_{\substack{H \in \mathscr{H} \\ H \subset F_2 }}  \sum_{\substack{Q \in \mathcal{S}^{\prime} \\ \pi(Q)= (F_2, H)}}\frac{\lambda_{Q}\chi_{Q}}{v_{\vec{w}}(Q)} \bigg\|_{L^{p}\left(v_{\vec{w}}\right)} & = \bigg\|\sum_{\substack{Q \in \mathcal{S}^{\prime} \\ \pi_{\mathscr{F}_2}(Q)= F_2 }}\frac{\lambda_{Q}\chi_{Q}}{v_{\vec{w}}(Q)} \bigg\|_{L^{p}\left(v_{\vec{w}}\right)} \nonumber\\ 
& \lesssim [\vec{w}]_{A_{\vec{P}}}^{\frac{1}{p}}\bigg(\sum_{\substack{Q \in \mathcal{S}^{\prime} \\ \pi_{\mathscr{F}_2}(Q)= F_2 }}\ave{\sigma_1}_Q^{\frac{p}{p_1}}\ave{\sigma_2}_Q^{\frac{p}{p_2}}|Q|\bigg)^{\frac{1}{p}}. \nonumber
\end{align}

Let $\varepsilon = \frac{1}{2^{11+d}[\sigma_1]_{A_{\infty}}}$, Hyt\"{o}nen and P\'{e}rez proved the reverse h\"{o}lder inequality $$\ave{\sigma_1^{1 + \varepsilon}}_{Q} \lesssim \ave{\sigma_1}^{1 + \varepsilon}_{Q},\quad \forall Q \subset \mathbb{R}^n$$ in \cite{HP}. Let $\gamma:=\frac{p}{p_1}\frac{1}{1+\varepsilon}$, $\eta:=\frac{p}{p_2}$, $\frac{1}{r}:=\gamma+\eta$, $\frac{1}{s}:=\gamma+\frac{1}{2}(1-\frac{1}{r})$, $\frac{1}{s^{\prime}}:=1-\frac{1}{s}$, we have
\begin{align}
I_1 & \lesssim [\vec{w}]_{A_{\vec{P}}}^{\frac{1}{p}} \bigg(\sum_{F_2 \in \mathscr{F}_2 }  \left(\ave{f_2}_{F_2}^{\sigma_2}\right)^{p}\sum_{\substack{Q \in \mathcal{S}^{\prime} \\ \pi_{\mathscr{F}_2}(Q)= F_2 }}\ave{\sigma_1}_Q^{\frac{p}{p_1}}\ave{\sigma_2}_Q^{\frac{p}{p_2}}|Q|\bigg)^{\frac{1}{p}} \nonumber \\
& \leqslant [\vec{w}]_{A_{\vec{P}}}^{\frac{1}{p}} \bigg(\sum_{F_2 \in \mathscr{F}_2 }  \left(\ave{f_2}_{F_2}^{\sigma_2}\right)^{p}\sum_{\substack{Q \in \mathcal{S}^{\prime} \\ \pi_{\mathscr{F}_2}(Q)= F_2 }}\ave{{\sigma_1}^{1+\varepsilon}}_Q^{\gamma}\ave{\sigma_2}_Q^{\eta}|Q|\bigg)^{\frac{1}{p}} \nonumber \\
& \leqslant [\vec{w}]_{A_{\vec{P}}}^{\frac{1}{p}} \bigg(\sum_{F_2 \in \mathscr{F}_2 }  \left(\ave{f_2}_{F_2}^{\sigma_2}\right)^{p}\bigg(\sum_{\substack{Q \in \mathcal{S}^{\prime} \\ \pi_{\mathscr{F}_2}(Q)= F_2 }}\ave{{\sigma_1}^{1+\varepsilon}}_Q^{s\gamma}|Q|\bigg)^{\frac{1}{s}}\bigg(\sum_{\substack{Q \in \mathcal{S}^{\prime} \\ \pi_{\mathscr{F}_2}(Q)= F_2 }}\ave{\sigma_2}_Q^{s^{\prime}\eta}|Q|\bigg)^{\frac{1}{s^{\prime}}}\bigg)^{\frac{1}{p}} \nonumber \\
& \leqslant [\vec{w}]_{A_{\vec{P}}}^{\frac{1}{p}} \bigg(\sum_{F_2 \in \mathscr{F}_2 }\sum_{\substack{Q \in \mathcal{S}^{\prime} \\ \pi_{\mathscr{F}_2}(Q)= F_2 }}\ave{{\sigma_1}^{1+\varepsilon}}_Q^{s\gamma}|Q|\bigg)^{\frac{1}{sp}} \times \bigg(\sum_{F_2 \in \mathscr{F}_2 }\left(\ave{f_2}_{F_2}^{\sigma_2}\right)^{s^{\prime}p}\sum_{\substack{Q \in \mathcal{S}^{\prime} \\ \pi_{\mathscr{F}_2}(Q)= F_2 }}\ave{\sigma_2}_Q^{s^{\prime}\eta}|Q|\bigg)^{\frac{1}{s^{\prime}p}}  \nonumber \\
& :=[\vec{w}]_{A_{\vec{P}}}^{\frac{1}{p}} J_1 \times J_2. \nonumber
\end{align}

Since $\mathcal{S}^{\prime}$ is sparse, for $J_1$, we have
\begin{align}
J_1 & \lesssim  \bigg(\sum_{F_2 \in \mathscr{F}_2 }\sum_{\substack{Q \in \mathcal{S}^{\prime} \\ \pi_{\mathscr{F}_2}(Q)= F_2 }}\ave{{\sigma_1}^{1+\varepsilon}}_Q^{s\gamma}|E_Q|\bigg)^{\frac{1}{sp}}  \nonumber \\
& \leqslant \Big(\int_{\widetilde{Q}}( M ({\sigma_1}^{1+\varepsilon}\chi_{\widetilde{Q}}))^{s\gamma}\ud x\Big)^{\frac{1}{sp}} \nonumber \\
& =\|M ({\sigma_1}^{1+\varepsilon}\chi_{\widetilde{Q}} )\|_{L^{s\gamma}\Big(\frac{\ud x}{|\widetilde{Q}|}\Big)}^{\frac{\gamma}{p}}|\widetilde{Q}|^{\frac{1}{sp}}  \nonumber \\
& \leqslant [\sigma_1]_{A_{\infty}}^{\frac{1}{sp}}\left\|M ({\sigma_1}^{1+\varepsilon}\chi_{\widetilde{Q}} )\right\|_{L^{1,\infty}\left(\frac{\ud x}{|\widetilde{Q}|}\right)}^{\frac{\gamma}{p}}|\widetilde{Q}|^{\frac{1}{sp}}  \nonumber \\
& \lesssim [\sigma_1]_{A_{\infty}}^{\frac{1}{sp}}\ave{{\sigma_1}^{1+\varepsilon}}_{\widetilde{Q}}^{\frac{\gamma}{p}}|\widetilde{Q}|^{\frac{1}{sp}} \lesssim [\sigma_1]_{A_{\infty}}^{\frac{1}{sp}}\ave{\sigma_1}_{\widetilde{Q}}^{\frac{1}{p_1}}|\widetilde{Q}|^{\frac{1}{sp}}.  \nonumber 
\end{align}

The third inequality in the above estimation is due to the Kolmogorov's inequality, that is, for any cube $Q$ in $\mathbb{R}^n$, $f \in L^{1,\infty}(Q)$, 
$$\|f\|_{L^{p}\left(\frac{\ud x}{|Q|}\right)} \leqslant \Big(\frac{1}{p}+\frac{1}{1-p}\Big)^{\frac{1}{p}}\|f\|_{L^{1,\infty}\left(\frac{\ud x}{|Q|}\right)},\quad  0 < p < \infty.$$
Specifically,
$$
\bigg(\frac{1}{s\gamma}+\frac{1}{1-s\gamma}\bigg)^{\frac{1}{sp}} = \bigg(\frac{1}{1-\frac{s}{2}\frac{\varepsilon}{1+\varepsilon}\frac{p}{p_1}}+\frac{2}{s}\frac{1+\varepsilon}{\varepsilon}\frac{p_1}{p}\bigg)^{\frac{1}{sp}} \lesssim [\sigma_1]_{A_{\infty}}^{\frac{1}{sp}}
$$
For $J_2$, using the same method as $J_1$ and (\ref{basic formula}), we obtain
\begin{align}
J_2 & \lesssim \Big(\sum_{F_2 \in \mathscr{F}_2 }\left(\ave{f_2}_{F_2}^{\sigma_2}\right)^{s^{\prime}p}[\sigma_1]_{A_{\infty}}\ave{\sigma_2}_{F_2}^{s^{\prime}\eta}|F_2|\Big)^{\frac{1}{s^{\prime}p}}\nonumber \\
& \leqslant [\sigma_1]_{A_{\infty}}^{\frac{1}{s^{\prime}p}} \Big(\sum_{F_2 \in \mathscr{F}_2}\left(\ave{f_2}_{F_2}^{\sigma_2}\right)^{p_2}\ave{\sigma_2}_{F_2}|F_2|\Big)^{\frac{1}{p_2}} \Big(\sum_{F_2 \in \mathscr{F}_2 }|F_2|\Big)^{\frac{1}{s^{\prime}p}-\frac{1}{p_2}}\nonumber \\
& \lesssim [\sigma_1]_{A_{\infty}}^{\frac{1}{s^{\prime}p}}\|f_2\|_{L^{p_2}\left(\sigma_2\right)}|\widetilde{Q}|^{\frac{1}{s^{\prime}p}-\frac{1}{p_2}}. \nonumber 
\end{align}
If we apply the reverse h\"{o}lder inequality for $\sigma_2$, we can obtain another bound similarly. Therefore, we get
$$I_1 \lesssim [\vec{w}]_{A_{\vec{P}}}^{1/p}\min\{[\sigma_1]_{A_{\infty}},[\sigma_2]_{A_{\infty}}\}^{1/p}\sigma_1(\widetilde{Q})^{1/p_1}\|f_2\|_{L^{p_2}\left(\sigma_2\right)}.$$
The estimation of $I_2$ is similar to $I_1$, only by replacing formula (\ref{key formula1}) with (\ref{key formula2}). By combining the above estimates of $I_1$ and $I_2$, we obtain the proof of the theorem.
\end{proof}

At the end of this section, we use Python to draw a graph to compare the weak type estimate we obtained with the sharp strong type estimate obtained by Li, Moen and Sun in \cite{LMS} mentioned in section 1. In particular, we show that when $p \geqslant \frac{3+\sqrt{5}}{2}$ or $\min\{p_1,p_2\} > 4$, the index we obtained is smaller than $1$.

\begin{figure}[htbp]
\centering
\includegraphics{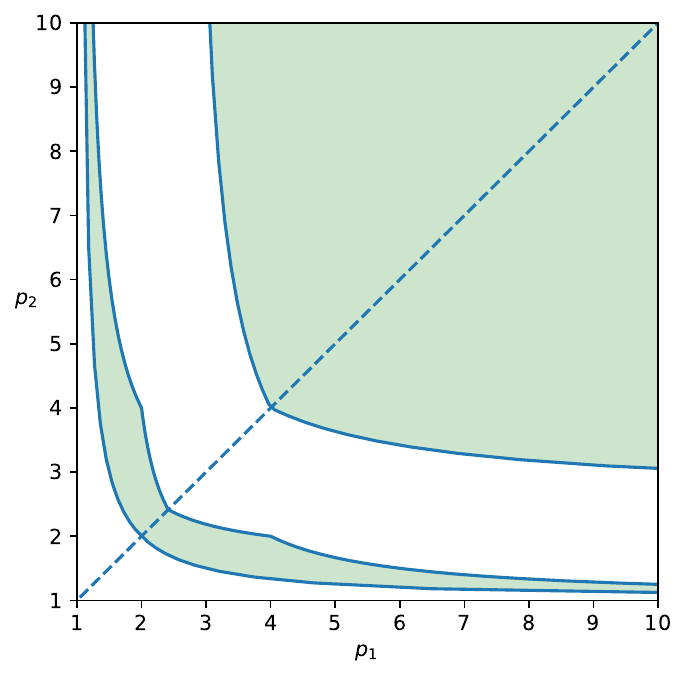}
\caption{Compared to the sharp strong type estimate, our results are better in shaded areas.}
\end{figure}

Without loss of generality, we assume $p_1 \leqslant p_2$ in the following caculations.

\begin{itemize} 
\item  When $p \geqslant \frac{3+\sqrt{5}}{2}$, it is obvious that $p_1^{\prime} \leqslant p$. In this case, the exponent in Theorem\ref{main theorem} is $\frac{1}{p} + \frac{1}{p_2^{\prime}}\frac{p_1^{\prime}}{p}$, if it is greater than or equal to $1$,we obtain 
$$
\frac{1}{p} + \frac{1}{p_2^{\prime}}\frac{p_1^{\prime}}{p} \geqslant 1 \,\Rightarrow\, \frac{p_1^{\prime}}{p_2^{\prime}} \geqslant p-1  \,\Rightarrow\, p_1^{\prime} \geqslant p-1 \,\Rightarrow\,  \frac{1}{p_1} \geqslant \frac{p-2}{p-1},
$$
Since $p \geqslant \frac{3+\sqrt{5}}{2}$, we have $\frac{p-2}{p-1} \geqslant \frac{1}{p}$, which leads to the contradiction.
\item  When $\min\{p_1,p_2\} > 4$, we can also obtain $p_1^{\prime} \leqslant p$, thus,
$$
\frac{1}{p} + \frac{1}{p_2^{\prime}}\frac{p_1^{\prime}}{p} = \frac{p_1^{\prime}}{p}\Big(2-\frac{1}{p}\Big) < 1  \,\Leftrightarrow\, \frac{2}{p}-\frac{1}{p^2} < \frac{1}{p_1^{\prime}},
$$
and this holds automatically since the left hand side is always less than $\frac{3}{4}$, while the right hand side is greater than it.
\end{itemize} 

\section{Acknowledgements}
The author thanks Professor Kangwei Li for suggesting this project, carefully reading the manuscript and providing many valuable suggestions, which greatly improve the quality of this article. Thanks also to Professor Sheldy Ombrosi for making some helpful comments on the paper.
\bibliographystyle{plain}
\bibliography{ref}

\begin{thebibliography}{10}

\bibitem{CR}
J.M. Conde-Alonso and G.~Rey.
\newblock A pointwise estimate for positive dyadic shifts and some
  applications.
\newblock {\em Math. Ann.}, 365(3-4):1111--1135, 2016.

\bibitem{DHL}
W.~Dami\'{a}n, M.~Hormozi, and K.~Li.
\newblock New bounds for bilinear {C}alder\'{o}n-{Z}ygmund operators and
  applications.
\newblock {\em Rev. Mat. Iberoam.}, 34(3):1177--1210, 2018.

\bibitem{Hyt}
T.~Hyt\"{o}nen.
\newblock The sharp weighted bound for general {C}alder\'{o}n-{Z}ygmund
  operators.
\newblock {\em Ann. of Math. (2)}, 175(3):1473--1506, 2012.

\bibitem{HL}
T.~Hyt\"{o}nen and M.T. Lacey.
\newblock The {$A_p$}-{$A_\infty$} inequality for general
  {C}alder\'{o}n-{Z}ygmund operators.
\newblock {\em Indiana Univ. Math. J.}, 61(6):2041--2092, 2012.

\bibitem{HP}
T.~Hyt\"{o}nen and C.~P\'{e}rez.
\newblock Sharp weighted bounds involving {$A_\infty$}.
\newblock {\em Anal. PDE}, 6(4):777--818, 2013.

\bibitem{Ler}
A.K. Lerner.
\newblock On an estimate of {C}alder\'{o}n-{Z}ygmund operators by dyadic
  positive operators.
\newblock {\em J. Anal. Math.}, 121:141--161, 2013.

\bibitem{LN}
A.K. Lerner and F.~Nazarov.
\newblock Intuitive dyadic calculus: the basics.
\newblock {\em Expo. Math.}, 37(3):225--265, 2019.

\bibitem{LOPTT}
A.K. Lerner, S.~Ombrosi, C.~P\'{e}rez, R.H. Torres, and
  R.~Trujillo-Gonz\'{a}lez.
\newblock New maximal functions and multiple weights for the multilinear
  {C}alder\'{o}n-{Z}ygmund theory.
\newblock {\em Adv. Math.}, 220(4):1222--1264, 2009.

\bibitem{Li}
B.~Li.
\newblock Some results on sparse operators.
\newblock {\em Master's thesis, Tianjin University}, 2023.

\bibitem{LMS}
K.~Li, K.~Moen, and W.~Sun.
\newblock The sharp weighted bound for multilinear maximal functions and
  {C}alder\'{o}n-{Z}ygmund operators.
\newblock {\em J. Fourier Anal. Appl.}, 20(4):751--765, 2014.

\bibitem{LS}
K.~Li and W.~Sun.
\newblock Weak and strong type weighted estimates for multilinear
  {C}alder\'{o}n-{Z}ygmund operators.
\newblock {\em Adv. Math.}, 254:736--771, 2014.

\bibitem{Zoe}
Z.~Nieraeth.
\newblock Quantitative estimates and extrapolation for multilinear weight
  classes.
\newblock {\em Math. Ann.}, 375(1-2):453--507, 2019.

\end{thebibliography}

\end{document}